\documentclass{amsart}
\usepackage{amsfonts,amssymb,amsmath,amsthm}
\usepackage{url}
\usepackage{enumerate}

\newtheorem{thrm}{Theorem}[section]

\newtheorem{prop}[thrm]{Proposition}
\newtheorem{cor}[thrm]{Corollary}
\theoremstyle{definition}
\newtheorem{definition}[thrm]{Definition}

\newtheorem{remark}[thrm]{Remark}
\newtheorem{example}[thrm]{Example}
\numberwithin{equation}{section}

\newcommand{\Ass}{\operatorname{Ass}}

\newcommand{\fgrade}{\operatorname{fgrade}}

\newcommand{\ara}{\operatorname{ara}}

\newcommand{\cd}{\operatorname{cd}}

\newcommand{\Ht}{\operatorname{ht}}
\newcommand{\pd}{\operatorname{pd}}

\newcommand{\Ext}{\operatorname{Ext}}

\newcommand{\dg}{\operatorname{dg}}

\newcommand{\Rad}{\operatorname{Rad}}

\newcommand{\depth}{\operatorname{depth}}

\newcommand{\vpl}{\operatornamewithlimits{\varprojlim}}

\newcommand{\vil}{\operatornamewithlimits{\varinjlim}}

\newcommand{\fm}{\frak{m}}
\newcommand{\fp}{\frak{p}}

\newcommand{\fa}{\frak{a}}

\author{Majid Eghbali}

\address{School of Mathematics, Institute for Research in Fundamental Sciences (IPM), P. O. Box: 19395-5746, Tehran-Iran.}

\email{m.eghbali@yahoo.com}
\thanks{This research was in part supported by a grant from IPM (No. 91130026)}


\begin{document}

\title[cohomologically complete intersection ideals]{On set theoretically and cohomologically complete intersection ideals}

\begin{abstract}
Let $(R,\fm)$ be a local ring and $\fa$ be an ideal of $R$. The  inequalities
     $$\begin{array}{ll} \   \Ht(\fa) \leq \cd(\fa,R) \leq \ara(\fa) \leq l(\fa) \leq \mu(\fa)
\end{array}$$
are known. It is an interesting and long-standing problem to find
out the cases giving equality. Thanks to the formal grade we give
conditions in which the above inequalities become
equalities.\end{abstract} \maketitle

\section{Introduction} \label{sect1}
Throughout this note, $R$ is a commutative Noetherian ring with identity and $\fa$ is an ideal of $R$.
 The smallest number  of  elements of $R$ required to generate $\fa$ up to radical is called the arithmetic rank, $\ara(\fa)$ of $\fa$.
Another invariant related to the ideal $\fa$ is $\cd(\fa,R)$, the so-called cohomological dimension of $\fa$, defined as the
 maximum index for which the local cohomology module $H^i_{\fa} (R)$ does not vanish.

It is well known that $\Ht(\fa) \leq \cd(\fa,R) \leq \ara(\fa).$ If
$\Ht(\fa) = \ara(\fa)$, $\fa$ is called a set-theoretic complete
intersection ideal. Determining set-theoretic complete intersection
ideals is a classical and long-standing problem in commutative
 algebra and algebraic geometry.
Many questions related to an ideal $\fa$ to being set-theoretic complete intersection are still open. See \cite{Lyu} for more information.

Recently, there have been many attempts to investigate the equality $\cd(\fa,R)=\ara(\fa)$, see  e.g. \cite{Ba1},
\cite{Ba2}, \cite{Ba3}, \cite{K-T-Y}  and their references, for certain classes of squarefree monomial ideals, but the
 equality does not hold in general (cf. \cite{Ya}). However, in many cases, this question is  open and many researchers are still working on it.

 Hellus and Schenzel \cite{H-Sch} defined an ideal $\fa$ to be cohomologically complete intersection, where  $\Ht(\fa) = \cd(\fa,R)$. In
 case $(R,\fm)$ is a Gorenstein local ring, they gave a characterization of cohomologically complete intersections for a certain class of ideals.

One more concept we will use is analytic spread of ideals. Let  $(R,\fm)$  be a local ring. We denote by $l(\fa)$, the Krull
 dimension of $\oplus^{\infty}_{n=0} (\fa^n/\fa^{n} \fm)$  called the analytic spread of $\fa$. In general,
$$\begin{array}{ll} \   \Ht(\fa) \leq \cd(\fa,R) \leq \ara(\fa) \leq l(\fa) \leq \mu(\fa),\ \ \ \ \ \ (\ast)
\end{array}$$
where $\mu(\fa)$ is the  minimal number of generators of $\fa$.
Burch in  \cite{Bu} proved what is now known as Burch's inequality
that $l(\fa) \leq \dim R- (\min_n \depth R/\fa^n)$. It should be
noted that the stability of $\depth R/\fa^n$ has been established by
Brodmann (cf. \cite{Br}). The equality $l(\fa)  = \dim R - (\min
\depth R/\fa^n)$
 has been  studied from several points of view by many authors and deep results have been obtained in recent
 years by   the assumptions that associated graded ring of $\fa$ being Cohen-Macaulay, see for instance
\cite[Proposition 3.3]{Ei-Hu} or \cite[Proposition 5.1]{Tr-Ik}  for detailed information.

The outline of this paper is as follows:

In Section two, we give a slight generalization of a result of Cowsik and Nori in order to turn some of the inequalities in  (*)
into the equality, (cf. Theorem \ref{2.8} ).
According to the results given in Section 2, one can see that the equality of $\Ht (\fa)=\cd(\fa,R)$ has a critical role in
clarifying the structure   of  $\fa$. In Sections  three and four  we have tried to focus our attention on this equality.

\section{Formal grade and depth} \label{ns}

Throughout this Section, $(R,\fm)$ is a commutative Noetherian local ring. Let $\fa$ be an ideal of $R$ and $M$ be an $R$-module. For an
 integer $i$, let $H^{i}_{\fa}(M) \ $ denote the $i$th local cohomology module of $M$. We have the isomorphism of $H^{i}_{\fa}(M) \ $  to
    ${\vil}_n \Ext^{i}_{R}(R/\fa^n, M)$ for every $i \in \mathbb {Z}$, see \cite{Br-Sh} for more details.

Consider the family of local cohomology modules $\{
H^{i}_{\fm}(M/{\fa}^n M) \}_{n \in \mathbb{N}} \ $. For every $n$
there is a natural homomorphism $H^{i}_{\fm}(M/{\fa}^{n+1} M)
\rightarrow H^{i}_{\fm}(M/{\fa}^n M)$ such that the family forms a
projective system. The projective limit
$\mathfrak{F}^i_{\fa}(M):={\vpl}_nH^i_{\fm}(M/\fa^n M)$ is called
the $i$th formal local cohomology of $M$ with respect to $\fa$ (cf.
\cite{Sch} also see \cite{As-D} and \cite{E} for more information).

For an ideal $\fa$ of $R$  the formal grade, $\fgrade(\fa,M)$, is defined as minimal non-vanishing of the formal cohomology modules, i.e.
$$\begin{array}{ll} \     \fgrade (\fa, M) = \inf \{i \in \mathbb {Z}:  {\vpl}_nH^i_{\fm}(M/\fa^n M) \neq 0 \}.
\end{array}$$

Formal grade is playing an important role throughout this note. We first recall a few remarks.

\begin{remark}\label{2.1} Let $\fa$ denote an ideal of  a local ring $(R,\fm)$. Let $M$ be a finitely generated $R$-module.
\begin{enumerate}
\item[(1)] $ \fgrade(\fa,M) \leq  \dim \widehat{R}/(\fa \widehat{R},\fp) , \text {\ for\  all \ } \fp \in \Ass \widehat{M}$
 (cf. \cite[Theorem 4.12]{Sch}).
\item[(2)] In case $M$ is a Cohen-Macaulay module $\fgrade(\fa,M)= \dim M - \cd (\fa,M)$ (cf. \cite[Corolary 4.2]{As-D}).
\end{enumerate}
\end{remark}

A key point in the proof of the main results in this section is the following:

\begin{prop} \label{2.2} Let $\fa$ be an  ideal of a $d$-dimensional local ring $(R,\fm)$. Then, the  inequality
$$\begin{array}{ll} \    \min_n \depth R/\fa^n  \leq \fgrade (\fa, R).
\end{array}$$
holds.
\end{prop}

\begin{proof}
Put $\min_n \depth R/\fa^n:=t$,  then for each integer $n$, $H^{i}_{\fm }(R/\fa^n )=0$ for all $i<t$. It implies that ${\vpl} H^i_{\fm}(R/\fa^n )=0$
for all $i<t$. Then,  the definition of the formal grade implies that $\min_n \depth R/\fa^n  \leq \fgrade (\fa, R)$.
\end{proof}

The above inequality may be strict as the next example demonstrates.

\begin{example}\label{2.3} Let $k$ be a field and $R=k[\left|x,y,z\right|]$ denote the formal power series ring in three variables over $k$.
Put $\fa := (x, y) \cap (y, z) \cap (x, z)$. Easily one can see that $\depth R/\fa=1$. However,
$\fa^2 = (x, y)^2 \cap (y, z)^2 \cap (x, z)^2 \cap (x^2, y^2, z^2)$ and consequently $\depth R/\fa^2=0$.

On the other hand, ${\vpl}_nH^0_{\fm}(R/\fa^n)=0$  then we have  $\ \fgrade (\fa, R)=1$.

\end{example}

In connection with the above results we state the next definition.

\begin{definition}\label{2.4} Let $\fa$ be an  ideal of a  local ring $(R,\fm)$. We define a non-negative integer $\dg(\fa)$ to measure
 the distance between the $\fgrade (\fa, R)$ and the lower bound of $\depth R/\fa^n, \ n \in \mathbb {N}$, i.e.
$$\begin{array}{ll} \   \dg(\fa):= \fgrade (\fa, R)- \min_n \depth R/\fa^n.
\end{array}$$
It should be noted that the stability of $\depth R/\fa^n$ has been
established by Brodmann (cf. \cite{Br}).
\end{definition}

Inspired of Remark \ref{2.1}, $\fgrade(\fa,R) \leq \dim (\widehat{R}/\fa \widehat{R}+\fp)$ for all $\fp \in \Ass \widehat{R}$. It
can be a suitable upper bound to control the formal grade of $\fa$ and $\min_n \depth R/\fa^n$ as well. It is clear that in case
 $\Rad(\fa \widehat{R}+\fp)=\fm \widehat{R}$ for some $\fp \in \Ass \widehat{R}$, then $\min_n \depth R/\fa^n  = \fgrade (\fa, R)=0$
  and consequently $\dg(\fa)=0$.

\begin{example}\label{2.5} Let $R = k[[x,y,z]]/(xy,xz)$ and $\fa:=(x,y)$. One can see that $\fgrade(\fa,R)=0$ and consequently $\dg(\fa)=0$.
\end{example}

\begin{prop}\label{2.6} Let $\fa$ be an  ideal of a  Cohen-Macaulay local ring $(R,\fm)$. Then, the following statements are true:
\begin{enumerate}
\item[(1)] If $\dg(\fa)=0$, then the following are equivalent:
\begin{enumerate}
\item[(a)] $\Ht(\fa)= \cd(\fa,R)$.
\item[(b)] $\fa$ \ is\ a\ set-theoretic\ complete\ intersection\ ideal.
\end{enumerate}
\item[(2)] Suppose that $\dg(\fa)= 1$, then
 $$\begin{array}{ll} \  l(\fa) \neq \dim R- \min_n \depth R/\fa^n \text{\ \ \  if\  and\  only\  if\ \ \  } \cd(\fa,R)= \ara (\fa)= l(\fa).
\end{array}$$

\end{enumerate}
\end{prop}

\begin{proof} \begin{enumerate}
\item[(1)] In case $\dg(\fa)=0$, then the inequalities
$$\begin{array}{ll} \   \Ht(\fa) \leq \cd(\fa,R) \leq l(\fa) \leq \dim R -\fgrade(\fa,R)
\end{array}$$
hold. Moreover, if $R$ is a Cohen-Macaulay ring, then in the light of Remark \ref{2.1} (2) and $(\ast)$ one has
$$\begin{array}{ll} \   \Ht(\fa) \leq \cd(\fa,R) & \leq l(\fa)
\\& \leq \dim R - \min_n \depth R/\fa^n
\\& =\dim R -\fgrade(\fa,R)
\\& = \cd(\fa,R) .
\end{array}$$
By virtue of $(\ast)$ and in conjunction with the above equalities, the statements "a" and "b" are equivalent.

\item[(2)] Assume that $l(\fa) \neq \dim R- \min_n \depth R/\fa^n$, then by assumptions we have
$$\begin{array}{ll} \  \cd (\fa,R) \leq \ara(\fa) \leq l(\fa) & <\dim R- \min \depth R/\fa^n
\\& = \dim R- \fgrade(\fa,R)+1
\\& = \cd(\fa,R)+1.
\end{array}$$
Now the claim is clear.

For the reverse implication, assume that $l(\fa) =\dim R- \min \depth R/\fa^n$.  If this is the case,
then $l(\fa) =\dim R- \fgrade (\fa,R)+1= \cd (\fa,R)+1$, which is a contradiction.
\end{enumerate}
\end{proof}

\begin{example}\label{2.7} Let $R = k[[x_1,x_2,x_3,x_4]]$ be the formal power series ring over a  field $k$ in four variables
 and $\fa= (x_1,x_2)\cap (x_3,x_4)$. Clearly one can see that $\dim R/\fa=2$, $\fgrade (\fa,R)=1$ and by virtue of
  \cite[Lemma 2]{Sch-V} $\min_n \depth R/\fa^n=1$, i.e. $\dg (\fa)=0$. On the other hand $\Ht (\fa)=2$ and $\cd (\fa,R)=3$. By
   a Mayer-Vietoris sequence one can see that $H^3_{\fa}(R) \neq 0$, that is  $\ara(\fa) =3= l(\fa)$.
\end{example}

For a prime ideal $\fp$  of  $R$,  the $n$th symbolic power of $\fp$
is denoted by $\fp^{(n)}= \fp^n R_{\fp} \cap R$. The
 following Theorem, gives  conditions at which the  required  equality of $(\ast)$ (in the Introduction)  is provided.

\begin{thrm} \label{2.8} Let $\fp$ be a  prime ideal of a Cohen-Macaulay local ring $(R,\fm)$ with  $\fgrade(\fp,R)\leq 1$. The
following statements are true:
\begin{enumerate}
\item[(1)] If $\fp^{(n)}= \fp^n$, for all $n$, then $l(\fp) =\cd(\fp,R) = \dim R - 1$.
\item[(2)] If $l(\fp)  = \dim R - 1$ and $\Ht (\fp)=\cd(\fp,R)$ then, $\fp$ is a set-theoretic complete intersection.
\end{enumerate}
\end{thrm}

\begin{proof} \begin{enumerate}
\item[(1)] As $\fp^{(n)}= \fp^n$, for all $n$, hence all of prime divisors of $\fp^n$ are minimal for all $n$, that is $\depth R/\fp^n >0$. On the
 other hand, Proposition \ref {2.2} implies that $\min \depth R/\fp^n= \fgrade (\fp,R)=1$, so it follows the claim. To this end note that
$$\begin{array}{ll} \  \cd (\fp,R) \leq l(\fp) & \leq \dim R-1
\\& = \dim R- \fgrade (\fp,R)
\\& = \cd(\fp,R).
\end{array}$$

\item[(2)] As  $\fgrade (\fp,R) \leq 1$ and $R$ is a Cohen-Macaulay local ring, then we have $\cd (\fp,R)=\dim R- \fgrade (\fp,R) \geq \dim R-1$.
 Hence,
$$\begin{array}{ll} \  \dim R -1 \leq \cd (\fp,R) =\Ht (\fp) \leq l(\fp) = \dim R-1.
\end{array}$$
It follows that $\fp$ is a set-theoretic complete intersection ideal.
\end{enumerate}
\end{proof}

Since $\fgrade (\fp,R) \leq \dim R/\fp$, one can get the following Corollary of Theorem \ref{2.8}.

\begin{cor} \label{2.9} Let $\fp$ be a one-dimensional  prime ideal of a Cohen-Macaulay local ring $(R,\fm)$. Then, $(1)$ implies $(2)$ and $(2)$
implies $(3)$.
\begin{enumerate}
\item[(1)] $\fp^{(n)}= \fp^n$, for all $n$.
\item[(2)] $l(\fp)  = \dim R - 1$.
\item[(3)]  $\fp$ is a set-theoretic complete intersection.
\end{enumerate}
\end{cor}

It should be noted that  Cowsik and Nori \cite[Proposition 3]{C-N} with some extra assumptions have shown that the conditions in
Corollary \ref{2.9} are equivalent for $\fp$ generated by an $R$-sequence.

\section{Case one: the ring of positive characteristic}

Let $p$ be a prime number and $R$ a commutative Noetherian ring of
characteristic $p$. The Frobenius endomorphism of $R$ is the map
$\varphi: R \rightarrow R$ where $\varphi(r)=r^p$. Let $\fa = (x_1,
... , x_n)$ be an ideal of  $R$. $\fa^{[p^e]}$ is the $e$th
Frobenius powers of $\fa$, defined by
  $\fa^{[p^e]} = (x^{p^e}_1,...,x^{p^e}_n )R$. Then, $\fa^{np^e} \subseteq
\fa^{[p^e]} \subseteq \fa^{p^e}$, i.e. $\fa^{[p^e]}$ and $ \fa^{p^e}$ have the same radical (cf. \cite{Bru-Her}).

Peskine and Szpiro \cite[Chap. 3, Proposition 4.1]{P-S} proved that for a  regular local ring $R$ of characteristic $p > 0$ and
 an ideal  $\fa$ of  $R$, if   $R/\fa$ is a Cohen-Macaulay ring, then $\Ht (\fa)= \cd (\fa, R)$. Below, (see Proposition \ref{3.2})
  we  give a generalization of their result.

\begin{remark} \label{3.1} Let $(R,\fm)$ be a regular local ring of characteristic $p >0$. Then, the following inequality holds:
$$\begin{array}{ll} \   \depth R/\fa \leq \fgrade(\fa,R) \leq \dim R/\fa.
\end{array}$$
\end{remark}

\begin{proof}
It is known that $\fgrade (\fa,R) \leq \dim R/\fa$ (cf. Section $2$).  By what we have seen above,
depth of $R/\fa$ is the same as the depth of every iteration of it. Put $l:= \depth R/\fa= \depth R/\fa^{[p^e]}$
for each integer $e$. It induces that $H^i_{\fm}(R/\fa^{[p^e]})$ is zero for all $i<l$, then so is
  ${\vpl} H^{i}_{\fm}(R/\fa^{p^e})= {\vpl} H^{i}_{\fm}(R/\fa^{[p^e]})$ (cf. \cite[Lemma 3.8]{Sch}).
  Hence, $l \leq \fgrade(\fa,R)$. Therefore we get the desired inequality.
\end{proof}

Note that in case $R$ is a Cohen-Macaulay local ring (not necessarily of positive characteristic), then in the light
 of \ref{2.1}(2) the following statement holds:
$$\begin{array}{ll} \   \Ht (\fa)= \cd (\fa, R) \ \text{\ if \ and\ only\ if\ } \fgrade(\fa,R)=\dim R/\fa.
\end{array}$$

\begin{prop}\label{3.2} Let $(R,\fm)$ be a regular local ring of characteristic $p >0$. Then, the following statements are equivalent:
\begin{enumerate}
\item[(1)] $R/\fa$ is a Cohen-Macaulay ring.
\item[(2)]  $\Ht (\fa)= \cd (\fa, R)$ and $H^{s}_{\fm}(R/\fa^{[p^{e+1}]}) \rightarrow H^{s}_{\fm}(R/\fa^{[p^e]})$ is
epimorphism for each integer $e$, where $s:=\depth R/\fa$.
\end{enumerate}
\end{prop}

\begin{proof}
$(1) \Rightarrow (2)$ As $R/\fa$ is a Cohen-Macaulay ring, then by assumption every iteration of $R/\fa$ is again  a Cohen-Macaulay
ring. Hence, $H^i_{\fm}(R/\fa^{[p^e]})$ is zero for all $i < \dim R/\fa$, then so is
${\vpl} H^i_{\fm}(R/\fa^{[p^e]}) \cong {\vpl} H^i_{\fm}(R/\fa^{p^e})$  for all  $i< \dim R/\fa$ (cf. \cite[Lemma 3.8]{Sch}).

 By virtue of \cite[Remark 3.6]{Sch}, one can see that  $H^{\dim R-i}_{\fa}(R)=0$ for all $\dim R-i >\Ht(\fa)$, i.e. $\Ht (\fa)= \cd (\fa, R)$.
 The second part of the claim follows by Hartshorne's non-vanishing Theorem, since $\depth R/\fa=\dim R/\fa$.

$(2) \Rightarrow (1)$ Assume that  $\Ht (\fa)= \cd (\fa, R)$. Then, $ \fgrade(\fa,R)= \dim R/\fa$. If we can prove that
 $\depth R/\fa \geq \fgrade(\fa,R)$, we would be done. Consider the epimorphism  of non-zero $R$-modules for each $e$:
$$\begin{array}{ll} \   H^s_{\fm}(R/\fa^{[p^{e+1}]}) \rightarrow H^s_{\fm}(R/\fa^{[p^{e}]}) \rightarrow 0.
\end{array}$$
Hence, \cite[Lemma 3.5.3]{We} implies that $\fgrade(\fa,R) \leq \depth R/\fa$. This completes the proof.
\end{proof}

\section{Case two: The polynomial ring}

Throughout this section, assume that $R=k[x_{1} ,..., x_{n}]$ is a polynomial ring in $n$ variables $x_{1} ,..., x_{n}$ over a field
$k$. Let $S:=k[x_{1} ,..., x_{n}]_{(x_{1} ,..., x_{n})}$ be the local ring  and $I$ be a square free monomial ideal of $S$:

\begin{prop}\label{4.1} Let $S$ and $I$ be as above. Then, the following are true.
\begin{enumerate}
\item[(1)]
$$\begin{array}{ll} \ H^{i}_{I }(S) = 0 \   \Longleftrightarrow \ {\vpl}_t H^{n-i}_{\fm}(S/I^t)=0\  \Longleftrightarrow \ H^{n-i}_{\fm }(S/I) = 0,
\end{array}$$
for a given integer $i$. In particular $S/I$ is a Cohen-Macaulay ring if and only if ${\vpl}_t H^{j}_{\fm}(S/I^t)=0 \ $ for all $j < n- \Ht I$.
\item[(2)] If $\Ht(\fa)=\cd(\fa,R)$ then,  $R/\fa$ is a Cohen-Macaulay ring, provided that $\fa$ is a squarefree monomial ideal of $R$.
\end{enumerate}
\end{prop}

\begin{proof}
\begin{enumerate}
\item[(1)]
By virtue of \cite[Remark 3.6]{Sch} $H^{i}_{I }(S) = 0 \ $  if and only if ${\vpl}_t H^{n-i}_{\fm}(S/I^t)=0 \ $. On the other hand by
virtue of \cite[Corollary   4.2]{S-W}, we have
$$\begin{array}{ll} \ H^{i}_{I }(S) = 0 \ \text{\ \ \  if and only if\ \ \ } \ H^{n-i}_{\fm }(S/I) = 0.
\end{array}$$
The second assertion easily follows by the first one.

\item[(2)]
 Without loss of generality, we may assume that $R$ is a  local ring with the graded maximal ideal $\fm=(x_{1} ,..., x_{n})$.
Now the claim follows by part (1).

\end{enumerate}
\end{proof}

The next result provides as a consequence an upper bound for the $\depth S/I^l$ for each $l \geq 1$.  Moreover, the second part of the
next result has been proved by Lyubeznik in \cite{Lyu2}.

\begin{cor} \label{4.2} Let $R$, $S$ and $I$ be as above.
\begin{enumerate}
\item[(1)] $\depth S/I= \fgrade(I,S)$ holds.
\item[(2)] Assume that $\fa$ is a squarefree monomial ideal in $R$. Then, $\pd_R R/\fa =\cd(\fa,R)$.
\end{enumerate}
\end{cor}

\begin{proof}
Assume that $\fgrade(I,S):=t$, then for all $i <t$ we have ${\vpl}_t H^{i}_{\fm}(S/I^t)=0$ if and only if  $H^{i}_{\fm }(S/I) = 0$
(cf. Proposition \ref{4.1}).
 Hence, $t \leq \depth S/I$. On the other hand assume that $\depth S/I:=s$. Again using Proposition \ref{4.1} we have $s \leq \fgrade(I,S)$,
  as desired.

In order to prove the second part, note that both $\pd_R R/\fa$ and  $\cd(\fa,R)$ are finite. Since
 $\pd R/\fa=\pd R_{\fm}/\fa R_{\fm}$ and $\cd(\fa,R)=\cd(\fa R_{\fm},R_{\fm})$, with $\fm=(x_1,...,x_n)$, then without loss of generality,
 we may assume that $R$ is a local ring with the homogeneous maximal ideal $\fm=(x_1,...,x_n)$. Now, by the Auslander-Buchsbaum formula and the first
 part one can get the claim. To this end note that
$$\begin{array}{ll} \ \pd_R R/\fa & = \depth R- \depth R/\fa
\\& = \dim R- \fgrade(\fa,R)
\\& =\cd(\fa,R).
\end{array}$$
\end{proof}

In the light of  Corollary \ref {4.2}, it is noteworthy to mention that for a squarefree monomial ideal $I$, we have
$$\begin{array}{ll} \ \depth S/I^l \leq  \fgrade(I,S)
\end{array}$$
for all positive integer $l$. Notice that $\depth S/I^l \leq \depth S/I$ for all positive integer $l$, see for example \cite{He-T-T}.

\begin{cor} \label{4.3} Let $R=k[x_{1} ,..., x_{n}]$ be a polynomial ring in $n$ variables $x_{1} ,..., x_{n}$ over a field $k$ and
$\fa$ be a squarefree monomial ideal of $R$. Then, the following are equivalent:
\begin{enumerate}
\item[(i)] $H^i_{\fa}(R)=0$ for all $i \neq \Ht \fa$, i.e. $\fa$ is cohomologically a complete intersection ideal.
\item[(ii)] $R/\fa$ is a Cohen-Macaulay ring.
\end{enumerate}
\end{cor}

\begin{proof}
Since each of the modules in question is graded, so the issue of vanishing is
unchanged under localization at the homogeneous maximal ideal of $R$. Hence, the claim follows by Proposition \ref{4.1}.
\end{proof}

Let $\overline{x_1},...,\overline{x_n}$ be the image of the regular
sequence $x_1,...,x_n$ in $S$. Let $k, l \leq n$ be arbitrary
integers. For all
 $i=1,...,k$ set $I_i:=(\overline{x_{i_1}},...,\overline{x_{i_{r_i}}})$, where the elements $\overline{x_{i_j}}$, $1 \leq j \leq r_i \leq l$ are
 from the set $\{ \overline{x_1},...,\overline{x_n}\}$ and a squarefree monomial ideal $I$ be as follows
$$\begin{array}{ll} \ I=I_1 \cap I_2 \cap ... \cap I_k,
\end{array}$$
where the set of basis elements of the $I_i$ are disjoint.

\begin{prop} \label{4.4} Let $I$ be as above. Then $\cd(I,S)=\sum^{k}_{i=1}r_i -k+1$ and in particular, $\dg (I)=0$.
\end{prop}

\begin{proof} By virtue of \cite[Lemma 2]{Sch-V}, $\depth S/I= \depth S- \sum^{k}_{i=1}r_i +k-1 $.
As $\depth S/I=\fgrade(I,S)=\dim S- \cd (I,S)$, (cf. \ref {4.2} and \ref {2.1} ) so the claim is clear.
\end{proof}

\proof[Acknowledgements]

I would like to thank Professor Markus Brodmann, Professor J\"urgen Herzog and Professor Moty Katzman for  useful discussions on
 $\min_n \depth R/\fa^n$ during preparation of this paper. My thanks are due to Professor Peter Schenzel for comments on Proposition
 \ref{4.1}. I am also grateful to the reviewer for suggesting
 several improvements of the manuscript.

\end{document}